\documentclass[12pt]{amsart}
\usepackage{amssymb,amsmath,amsthm,comment}
\usepackage{pdfpages,graphicx} 
\usepackage{setspace}
\newcommand{\shrinkmargins}[1]{
  \addtolength{\textheight}{#1\topmargin}
  \addtolength{\textheight}{#1\topmargin}
  \addtolength{\textwidth}{#1\oddsidemargin}
  \addtolength{\textwidth}{#1\evensidemargin}
  \addtolength{\topmargin}{-#1\topmargin}
  \addtolength{\oddsidemargin}{-#1\oddsidemargin}
  \addtolength{\evensidemargin}{-#1\evensidemargin}
  }
\shrinkmargins{0.5}

\newtheorem{theorem}{Theorem}

\newtheorem{corollary}[theorem]{Corollary}

\newtheorem{proposition}[theorem]{Proposition}
{Claim}

\newtheorem*{definition}{Definition}

\theoremstyle{remark}

\newtheorem*{remark}{Remark}
\newtheorem*{example}{Example}
\numberwithin{theorem}{section} \numberwithin{equation}{section}

\usepackage[font=small,skip=0pt]{caption}
\usepackage{multirow}

\usepackage[section]{placeins} % figures should be displayed in same section
\usepackage{wrapfig}

\begin{document}
\title[Bessenrodt--Ono Inequality]{Polynomization of the \\Bessenrodt--Ono inequality}
%\author{}
%\address{}
%\email{}  

%\author{}
%\address{School of Mathematics\\University of \\  \\}
%\email{}  
\author{Bernhard Heim} 
\author{Markus Neuhauser}
\author{Robert Tr{\"o}ger}
%\thanks{%
%The authors thank the RWTH Aachen University and the Graduate School: 
%Experimental and constructive algebra for their excellent working environment.}
\address{Faculty of Science, German University of Technology in Oman, Muscat, Sultanate of Oman, \newline Faculty of Mathematics, Computer Science, and Natural Sciences, RWTH Aachen University,
52056 Aachen, Germany}
\email{bernhard.heim@gutech.edu.om}
\email{markus.neuhauser@gutech.edu.om}
\email{robert@silva-troeger.de}


%%%%\dedicatory{Preliminary}

%%\date{\today}
\begin{abstract}
In this paper we investigate the generalization of
the Bessenrodt--Ono inequality by following Gian-Carlo Rota's advice in studying
problems in combinatorics and number theory in terms of roots of polynomials.
We consider the number of $k$-colored partitions of $n$ as special values of 
polynomials $P_n(x)$.  We prove for all real numbers $x >2 $ and $a,b \in \mathbb{N}$ with $a+b >2$ the inequality
\begin{equation*}
P_a(x) \, \cdot \, P_b(x) > P_{a+b}(x).
\end{equation*}
%%
%%Here $x=3$ does not fulfill the inequation. 
We show that $P_n(x) < P_{n+1}(x)$ for $x \geq 1$, which generalizes $p(n) < p(n+1)$, where $p(n)$ denotes the partition function.
Finally, we observe for small values, the opposite can be true since for example: $P_2(-3+ \sqrt{10}) = P_{3}(-3 + \sqrt{10})$.
\end{abstract}

\maketitle
\newpage
\section{Introduction and main results}
Let $p(n)$ be the number of partitions of $n$ (\cite{An98,On03,AE04}).
It is well-known that this arithmetic function increases strictly:
%%is strictly increasing:
\begin{equation}
p(1) \, p(n) < p(n+1) \label{eq:Einstieg}
\end{equation}
for all $n \in \mathbb{N}$, since every partition of $n$ can be lifted to a partition of $n+1$.
Bessenrodt and Ono \cite{BO16} discovered that (\ref{eq:Einstieg}) is actually an exception in the context of
more general products $p(a)\,p(b)$. Recently, DeSalvo and Pak \cite{DP15} proved that the sequence $\{p(n)\}$ is log-concave for $n>25$.
\begin{theorem}[Bessenrodt--Ono 2016] \label{Ono} \ \\
Let $a,b$ be natural numbers. Let $a+b >9$, then
\begin{equation}
p(a) \, p(b) > p(a+b). \label{eq:BO}
\end{equation}
\end{theorem}
Bessenrodt and Ono provided proof, 
%%an analytic proof, 
based on a theorem of Rademacher \cite{Ra37} and Lehmer \cite{Le39}.
They speculated at the end of their paper, that combinatorial proof could be possible.
Shortly after their paper was published, Alanazi, Gagola, and Munagi \cite{AGM17} found such a proof.
Chern, Fu, and Tang \cite{CFT18} generalized Bessenrodt and Ono's theorem to $k$-colored partitions $p_{-k}(n)$ of
$n$.
\begin{theorem}[Chern, Fu, Tang 2018] \label{Fu} \ \\
Let $a,b,k$ be natural numbers. Let $k>1$, then
\begin{equation} \label{k=2}
p_{-k}(a) \,  p_{-k}(b)    >     p_{-k}(a+b), 
\end{equation}
except for $(a,b,k) \in \left\{ (1,1,2), (1,2,2), (2,1,2), (1,3,2), (3,1,2), (1,1,3)\right\} $.
\end{theorem}

Let $P_n(x)$ be the unique polynomial of degree $n$ satisfying $P_n(k)=p_{-k}(n)$ for all $k \in \mathbb{N}$ (see \cite{HLN19}).
In this paper we prove the following results:
\begin{theorem} \label{HNT Theorem}
Let $n \in \mathbb{N}$ and $x \in \mathbb{R}$ with $x \geq 1$. Then
\begin{equation}
P_n(x) < P_{n+1}(x) \text{ and } 1 \leq P_n'(x)< P_{n+1}'(x). \label{eq:comparison}
\end{equation}
\end{theorem}
\begin{remark} 
Let $n+1$ be a prime number then there exists $x_n \in (0,1)$, such that $$P_{n+1}(x_n)<P_n(x_n).$$
%%There exists an $x_0 \in \mathbb{R}$ with $0 < x_0 <1/2$ and $n \in \mathbb{N}$ such that the opposite is true:
%%$P_n(x_0)> P_{n+1}(x_0)$. This seems to be unexpected.
\end{remark}
%%The outcome of our investigation is the following.

Our main result is the following extension of the Bessenrodt--Ono type inequality. 
\begin{theorem}\label{Hauptsatz}
Let $a,b \in \mathbb{N}$, $a+b>2$, and $x >2$. Then
\begin{equation}
P_a(x) \, P_b(x)  >  P_{a+b}(x).  \label{conjecture}
\end{equation}
The case $x=2$ is true for $a+b>4$.
\end{theorem}
%%Note that the case $x=2$ and $a+b >4$ is contained in Theorem \ref{Fu}.
The theorem is proven by induction, using a special formula for the derivative of $P_n(x)$, 
the inequality (\ref{k=2}) for $k=2$ is proven by Chern, Fu, and Tang and
Theorem \ref{HNT Theorem}. This gives a precise answer to a conjecture stated in \cite{HN19}.

Bessenrodt--Ono type inequalities also appeared in work by Beckwith and Bessenrodt \cite{BB16} on $k$-regular partitions
and Hou and Jagadeesan \cite{HJ18} to the numbers of partitions with ranks in a given residue class modulo $3$.
Dawsey and Masri \cite{DM19} obtained new results for Andrews {\it{spt}}-function.
It is very likely that some of these recent results can be extended to an inequality of certain polynomials.
%%\newpage
%%
%%
%%
\section{Partitions and polynomials}
A partition $\lambda$ of a positive integer $n$ is any non-increasing sequence $\lambda_1,\lambda_2, \ldots, \lambda_d$ of positive integers
whose sum is $n$. The $\lambda_i$ denote the parts of the partition. The number of partitions of $n$ is denoted by $p(n)$
(see \cite{An98, On03}).
%%For convenience $p(0):=1$ (see \cite{An98, On03}).
\begin{example}
The partitions of $n=4$ and $n=5$ are
\begin{eqnarray*}
4 &=& 3+1=2+2=2+1+1=1+1+1+1\\
5 &=& 4+1=3+2=3+1+1=2+2+1=2+1+1+1=1+1+1+1+1.
\end{eqnarray*}
Hence $p(4)=5$ and $p(5)=7$. Note that $p(200)$ is already equal to $3972999029388$.
\end{example}
A partition is called a $k$-colored partition of $n$ if each part can appear in $k$ colors. 
Let $p_{-k}(n)$ denote the number of $k$-colored partitions of $n$ (see \cite{CFT18}, introduction).
Note that $p_{-1}(n)= p(n)$ and $p_{-k}(n) < p_{-(k+1)}(n)$. For example
$p_{-2}(4)= 20$ and $p_{-2}(5) =36$.
The generating function of $p_{-k}(n)$ is given by
\begin{equation}
\sum_{n=0}^{\infty} p_{-k}(n) \, q^n = \frac{1}{ \prod_{n=1}^{\infty} \left( 1 - q^n \right)^k } = \frac{1}{(q;q)_{\infty}^k}
\qquad (k \in \mathbb{N}).
\end{equation}
Here $(a;q)_{\infty}= \prod_{n=0}^{\infty} (1 - a \, q^n)$ (standard notion).

\begin{definition}
We define recursively a family of polynomials $P_n(x)$. Let $P_0(x):=1$ and
\begin{equation}
P_n(x) := \frac{x}{n} \sum_{k=1}^{n} \sigma(k) \, P_{n-k}(x).
\end{equation}
Here $\sigma(n) := \sum_{d \mid n} d$ denotes the sum of divisors of $n$.
\end{definition}
Then it is known that $P_n(k) = p_{-k}(n)$ for $k \in \mathbb{N}$ ( see \cite{HLN19}). Let $p_{-k}(0):=1$.
If we put $q:= e^{2 \pi i \tau}$ with
$\tau$ in the upper complex half-space, then
\begin{equation*}
\sum_{n=0}^{\infty} P_n(z) \,\, q^n = \prod_{n=1}^{\infty} \left( 1 - q^n \right)^{-z} \qquad (z \in \mathbb{C}). 
\end{equation*}
We have $P_0(x)=1$, $P_1(x) = x$, $P_2(x) = x/2 \, (x+3)$, $P_3(x) = x/6 \, (x^2 + 9x +8)$.
%%\newpage
\section{Basic properties of $P_n(x)$ and the proof of Theorem \ref{HNT Theorem}}
In this section we study the difference function
\begin{equation}
\Delta_n(x):= P_{n+1}(x) - P_n(x)
\end{equation}
to prove Theorem \ref{HNT Theorem}.
We first observe that
\begin{equation*}
\lim_{x \to \infty} \Delta_n(x) = + \infty.
\end{equation*}
This is true since $P_n(x)$ is a polynomial of degree $n$ of leading coefficient $1/n!$.
We have $P_n(x) = x/n!  \, \cdot \, \widetilde{P}_n(x)$, 
where $\widetilde{P}_n(x)$ is a normalized polynomial of degree $(n-1)$ with
positive integer coefficients. Hence, we also deduce that $\Delta_n(0)=0$. 
We are especially interested in the non-negative largest real root of $\Delta_n(x)$.
The real roots of $\Delta_1(x)$ are $\{-1,0\}$ 
and the real roots of $ \Delta_2(x)$ are $$\{-3-\sqrt{10}, \,0, \,  -3+ \sqrt{10}  \}.$$          
We can see already that $\Delta_n(x)$ is not always positive, here $ \Delta_2 ( x) <0$ for $$ 0 < x < -3+ \sqrt{10}.$$
%%
%%  Bild
%\begin{minipage}{0.65\linewidth}
%%\begin{proof}
\begin{wrapfigure}{r}{0.5\textwidth}
  \centering
  \vspace{-\baselineskip}
  \includegraphics[width=0.52\textwidth]{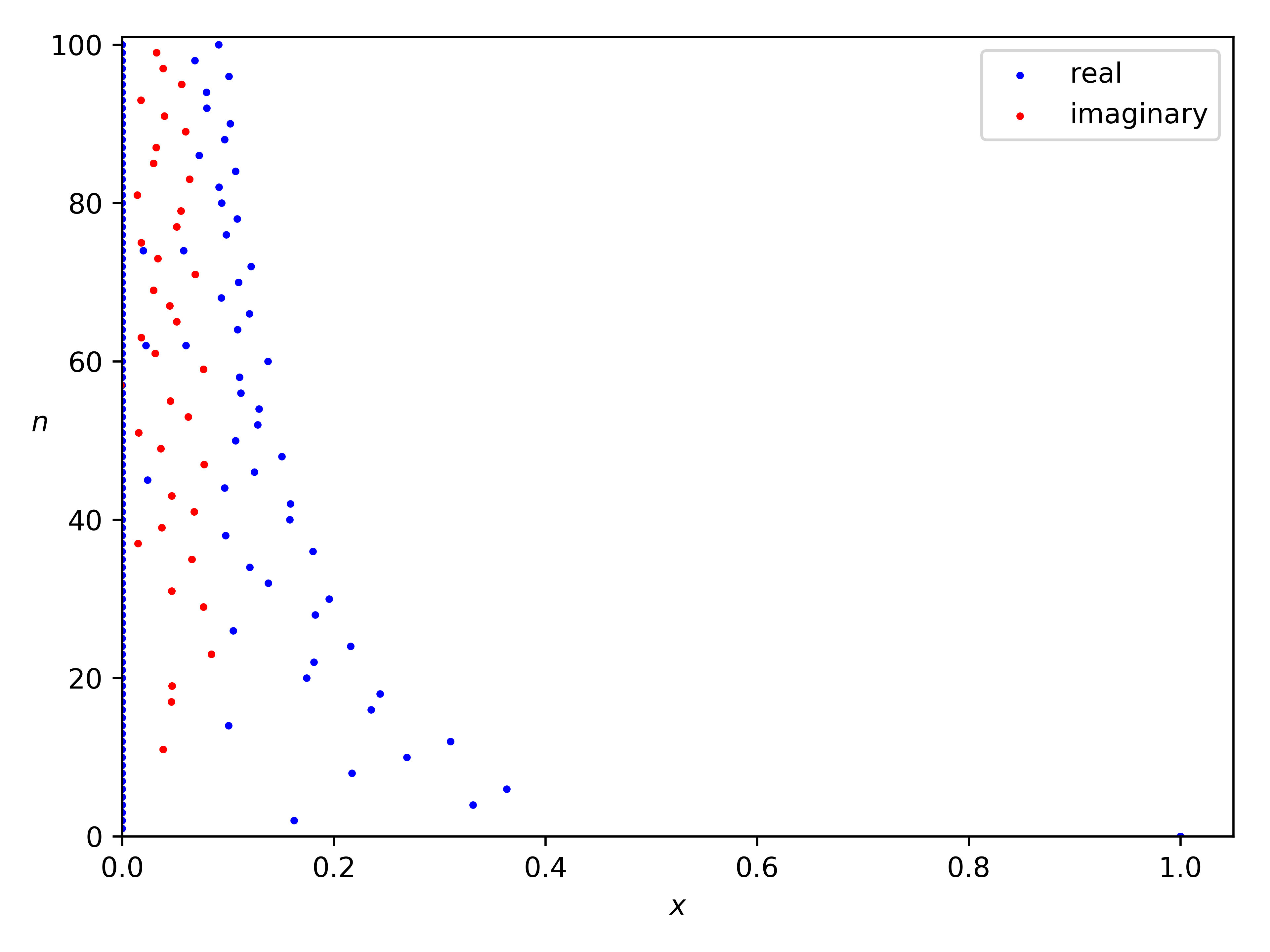}
  \caption{Roots of $\Delta_{n}(x)$ with positive real part---real part displayed.}
\end{wrapfigure}

%\end{minipage}
%%
%%

We prove the first part of Theorem \ref{HNT Theorem} by induction.
Claim:
$P_{n+1}(x) > P_n(x)$ for all $n \in \mathbb{N}$ and $x \geq 1$.
Let $n=1$. We have $\Delta_1(x)> 0$ for all $x>0$.
Suppose that $\Delta_m(x)>0$ is true for all real numbers $x \geq 1$ and integers $1 \leq m \leq n-1$.
\newline
\newline
\newline
\newline
\newline
\newline
\newline
There is a useful formula for the derivatives 
$P_n'(x)$ for all $n \in \mathbb{N}_0$ \cite{HN18}:
\vspace{1.5cm}
\begin{equation}
P_n'(x) = \sum_{k=1}^n \frac{\sigma(k)}{k} \, P_{n-k}(x).
\end{equation}
For $P_{n+1}'(x)$ we obtain the strict lower bound
\begin{equation}
\sum_{k=1}^n \frac{\sigma(k)}{k} \, P_{n+1-k}(x) \geq \sum_{k=1}^n \frac{\sigma(k)}{k} \, P_{n-k}(x). 
\end{equation}
Hence $P_{n+1}'(x) > P_{n}'(x)$. Property (\ref{eq:Einstieg}) for the partition function provides
\begin{equation}
P_{n+1}(1) = p(n+1)> p(n) = P_n (1).
\end{equation}
As a corollary we obtain that $\Delta_n'(x) >0$.
%%\end{proof}

We finally prove the remark in the introduction. Since $\Delta_n(0)=0$ and 
\begin{equation*}
\lim_{x \to \infty} \Delta_n(x) = + \infty
\end{equation*}
it is sufficient to show that $\Delta_n'(0) <0$ for $n+1$ a prime.
We have 
\begin{equation}
\Delta_n'(0) = \frac{\sigma(n+1)}{n+1} - \frac{\sigma(n)}{n} = \frac{n+2}{n+1} - \frac{\sigma(n)}{n} <0.
\end{equation}
Here $\sigma(n) > n+1$.
\section{Bessenrodt--Ono type inequality (BO)}
Let $a,b \in \mathbb{N}$. The Bessenrodt--Ono inequality $p(a) \, p(b) > p(a+b)$ 
is always satisfied for all $a,b \geq 2$ and $a+b >9$. Since the inequality is symmetric in $a$ and $b$ we assume $a \geq b$.
It was also shown \cite{BO16} that there is equality for $(a,b) \in \left\{ (6,2), (7,2), (4,3)\right\} $. 
The inequality fails completely for $b=1$ and $$(a,b) \in \left\{ (2,2),(3,2), (4,2), (5,2), (3,3),(5,3)\right\},$$ 
while it is true for the remaining cases $(a,b)\in \left\{ (4,4),(5,4)\right\} $.

The BO for $2$-colored partitions is true for all $a,b \in \mathbb{N}$ except $(a,b)=(1,1)$, 
where $p_{-2} (1) \, p_{-2}(1) < p_{-2}(2)$. Let $a \geq b$.
Then we have equality $(a,b) \in \{ (2,1),(3,1)\}$. 
The BO for $3$-colored partitions holds for all $a,b \in \mathbb{N}$ except for $(a,b)=(1,1)$, where we have equality.
If $k \geq 4$, then BO is fulfilled for all $a,b \in \mathbb{N}$ (see \cite{CFT18}).

Let $P_{a,b}(x):= P_a(x) \, P_b(x) - P_{a+b}(x)$. Then $P_{a,b}(0)= 0$, $P_{a,b}'(0) = - \sigma(a+b)/(a+b)$ and
\begin{equation*}
\lim_{ x \rightarrow \infty} P_{a,b}(x) = \infty.
\end{equation*}
In contrast to $\Delta_n(x)$, the polynomials $P_{a,b}(x)$ appear to have only one root $x_{a,b}$ with a positive real part.
The following table records these roots for $1 \leq a,b \leq 10$.
%%\newpage

\vspace{0.5cm}

% Table with positive real roots of P_{a,b}
\begin{tabular}{|c||cccccccccc|}
\hline
$x_{a,b}$&     1&     2&     3&     4&     5&     6&     7&     8&     9&    10 \\
\hline
\hline
    1 &     3.00 &     2.00 &     2.00 &     1.69 &     1.74 &     1.57 &     1.59 &     1.50 &     1.51 &     1.45 \\
    2 &     2.00 &     1.40 &     1.25 &     1.13 &     1.09 &     1.00 &     1.00 & \textbf{    0.95} & \textbf{    0.92} & \textbf{    0.91} \\
    3 &     2.00 &     1.25 &     1.24 &     1.00 &     1.05 & \textbf{    0.90} & \textbf{    0.94}& \textcolor{gray}{    0.85}& \textcolor{gray}{    0.87}& \textcolor{gray}{    0.81} \\
    4 &     1.69 &     1.13 &     1.00 & \textbf{    0.87} & \textbf{    0.86}& \textcolor{gray}{    0.76}& \textcolor{gray}{    0.76}& \textcolor{gray}{    0.72}& \textcolor{gray}{    0.69}& \textcolor{gray}{    0.67} \\
    5 &     1.74 &     1.09 &     1.05 & \textbf{    0.86}& \textcolor{gray}{    0.88}& \textcolor{gray}{    0.75}& \textcolor{gray}{    0.79}& \textcolor{gray}{    0.70}& \textcolor{gray}{    0.71}& \textcolor{gray}{    0.67} \\
    6 &     1.57 &     1.00 & \textbf{    0.90}& \textcolor{gray}{    0.76}& \textcolor{gray}{    0.75}& \textcolor{gray}{    0.66}& \textcolor{gray}{    0.66}& \textcolor{gray}{    0.60}& \textcolor{gray}{    0.60}& \textcolor{gray}{    0.57} \\
    7 &     1.59 &     1.00 & \textbf{    0.94}& \textcolor{gray}{    0.76}& \textcolor{gray}{    0.79}& \textcolor{gray}{    0.66}& \textcolor{gray}{    0.69}& \textcolor{gray}{    0.62}& \textcolor{gray}{    0.63}& \textcolor{gray}{    0.58} \\
    8 &     1.50 & \textbf{    0.95}& \textcolor{gray}{    0.85}& \textcolor{gray}{    0.72}& \textcolor{gray}{    0.70}& \textcolor{gray}{    0.60}& \textcolor{gray}{    0.62}& \textcolor{gray}{    0.56}& \textcolor{gray}{    0.55}& \textcolor{gray}{    0.53} \\
    9 &     1.51 & \textbf{    0.92}& \textcolor{gray}{    0.87}& \textcolor{gray}{    0.69}& \textcolor{gray}{    0.71}& \textcolor{gray}{    0.60}& \textcolor{gray}{    0.63}& \textcolor{gray}{    0.55}& \textcolor{gray}{    0.56}& \textcolor{gray}{    0.52} \\
   10 &     1.45 & \textbf{    0.91}& \textcolor{gray}{    0.81}& \textcolor{gray}{    0.67}& \textcolor{gray}{    0.67}& \textcolor{gray}{    0.57}& \textcolor{gray}{    0.58}& \textcolor{gray}{    0.53}& \textcolor{gray}{    0.52}& \textcolor{gray}{    0.49} \\
\hline
\end{tabular}
\ \newline \
\captionof{table}{Positive real roots of $P_{a,b}(x)$}
\ \newline
\newline
%%
%%There appears to be only one positive real root, starting at $x = 3$ and going near to $x = 1$ if $a$ goes to infinity.
%%\newpage
%%%%%%%%%%% Bild
\vspace{0.5cm}
%{%
  %\setlength\intextsep{0pt}
  %\setlength\parindent{0pt}
  %\noindent
  \begin{wrapfigure}{r}{0.5\textwidth}
    \centering
    \vspace{-\baselineskip}
    \includegraphics[width=0.48\textwidth]{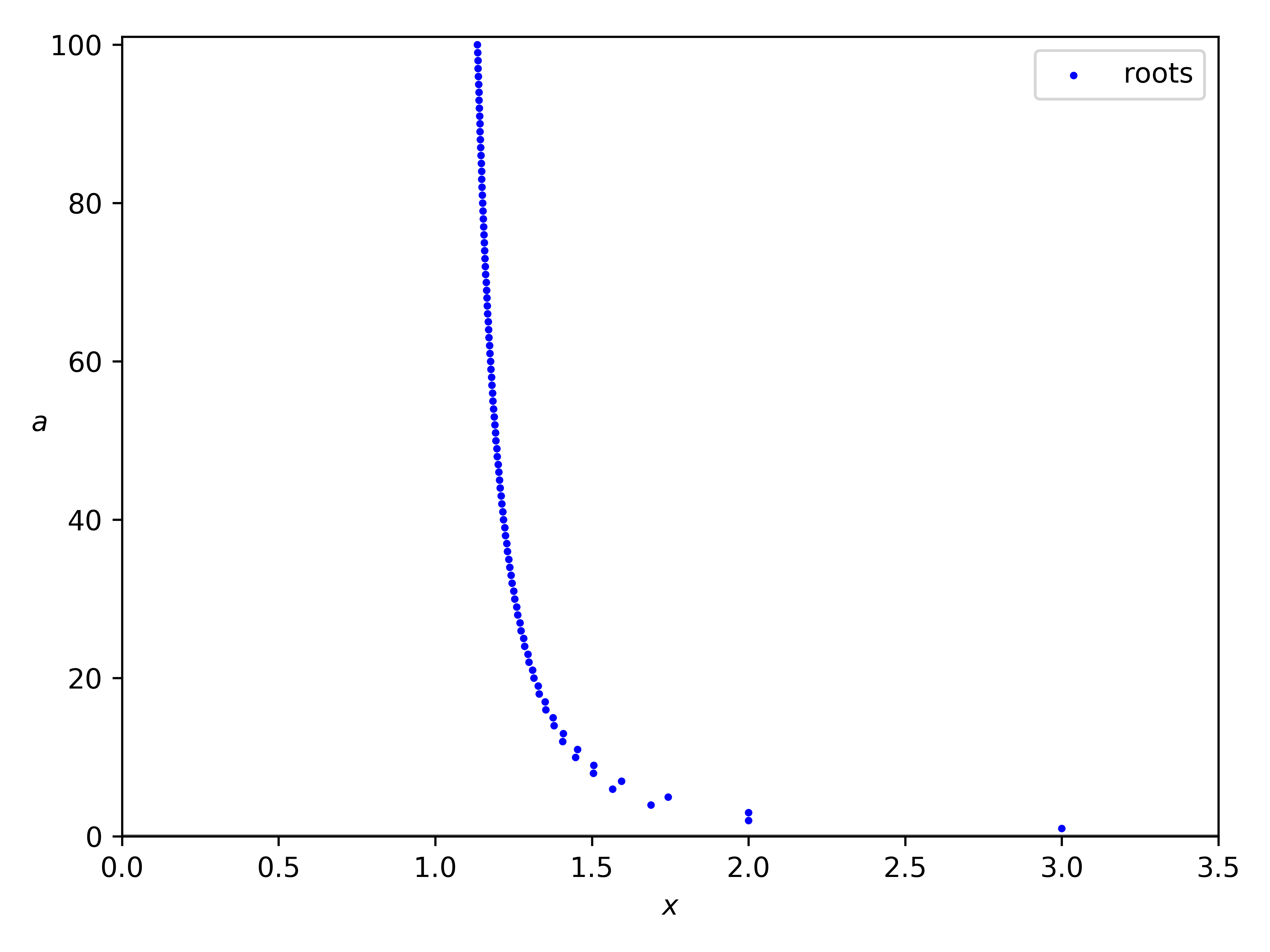}
    \caption{Positive real roots of $P_{a,1}(x)$.}
  \end{wrapfigure}
%}
 \ \newline
%%\newline
The root distribution in the table explains all exceptions which appear in the papers \cite{BO16} and \cite{CFT18}.
In Figure 2 we have displayed the single positive root $x_{a,1}$ of $P_{a,1}(x)$ for $1 \leq a \leq 100$. 
It seems that, in general, $P_{a,1}(x)$ has exactly one positive real root (and no non-real roots) and that the limit exists and is equal to $1$.
\section{Special Case}
Let $\mu(n):= \frac{\pi}{6} \sqrt{24n-1}$. Rademacher \cite{Ra37} proved the formula
\begin{equation*}
p(n) = \frac{\sqrt{12}}{24n-1}  \sum_{k=1}^{N} A_k^{*}(n)  \, 
\left( \left(1-\frac{k}{\mu}\right)  e^{\mu /k } + \left(1+ \frac{k}{\mu}\right)  e^{- \mu /k } \right) + R_2(n,N).
\end{equation*}
Here $A_k^{*}(n) = \frac{1}{\sqrt{k}} A_k(n)$ are real numbers, where $A_k(n)$ is a complicated sum of $24k$th roots of unity.
Lehmer (\cite{Le37, Le38}, \cite{Le39}, introduction) obtained the estimate
\begin{equation}
\vert R_2(n,N) \vert < \frac{\pi^2 N^{-2/3}}{\sqrt{3}} \left( \left(\frac{N}{\mu}\right)^{3} \, \text{sinh}\left(\frac{\mu}{N}\right) + \frac{1}{6} - \left(\frac{N}{\mu}\right)^{2} \right).
\end{equation}
for all $n, N \in \mathbb{N}$.
DeSalvo and Pak \cite{DP15} recently utilized the case $N=2$ and proved that $p(n)$ is log-concave for all $n>25$.
They proved two of Chen's conjectures. For our purpose the case $N=1$, which was studied by
Bessenrodt and Ono \cite{BO16}, is more convenient, 
They obtained:
\begin{equation}
\label{BO2}
\frac{\sqrt{3}}{12m} \left( 1 - \frac{1}{\sqrt{m}}\right) \mathrm{e}^{\frac{\pi }{6}\sqrt{ 24m-1}} < P_{m}\left( 1\right) < 
\frac{\sqrt{3}}{12m} \left( 1 + \frac{1}{\sqrt{m}}\right) \mathrm{e}^{\frac{\pi }{6}\sqrt{24m-1}}. 
\end{equation}
In the following we utilize the well-known upper bound:
$$\sigma \left( m\right) \leq  m \left( 1+\ln \left( m\right) \right).$$
%%%%%%%%%%%%%%%
%%%%%%%%%%%%%%%
%%%%%%%%%%%%%%%          Satz zu P_{n,1}(x).
%%%%%%%%%%%%%%%
\begin{proposition}
\label{monotonie}
Let $n \in \mathbb{N}$ and $x \in \mathbb{R}$. The inequality
\begin{equation}
x \,  P_{n}\left( x\right) >P_{n+1}\left( x\right) 
\end{equation}
holds for all $x>2$ and $n\geq 2$. In the case $n=1$ it holds for $x>3$.
\end{proposition}
%%%%%%%%
%%%%%%%%
\begin{proof}
The case $n=1$ is easy to see, since $P_{1,1}(x) = x/2 \, (x-3)$.
We prove the proposition by induction on $n$.
Let $n=2$. Then $P_{2,1}(x) = x/3 \, (x^2-4)>0$ for $x >2$.
If $n\geq 3$ and that $P_{m,1}(x) >0$ for $2 \leq m \leq n-1$ and $x>2$.

It is sufficient to prove that $\frac{\mathrm{d}}{\mathrm{d}x} P_{n,1}(x) > 0$, since 
we already know that $P_{n,1}(2) \geq 0$ (\cite{CFT18}, see introduction, Theorem 1.2).
The derivative of $P_{n,1}(x)$ is equal to
%%%%%$xP_{0}\left( x\right) =x=P_{1}\left( x\right) $.
\begin{eqnarray*}
& & P_{n}\left( x\right) +\sum _{k=1}^{n}\frac{\sigma \left( k\right) }{k}xP_{n-k}\left( x\right) -\sum _{k=1}^{n+1}\frac{\sigma \left( k\right) }{k}P_{n+1-k}\left( x\right) \\
&>&P_{n}\left( x\right) +\frac{\sigma \left( n-1\right) }{n-1}\left( \left( x-3\right) x/2\right) -\frac{\sigma \left( n+1\right)  }{n+1} \\
&\geq &P_{n}\left( x\right) -\left( 1+\ln \left( n-1\right) \right) \frac{9}{4}-\left( 1+\ln \left( n+1\right) \right) \\
&\geq &P_{n}\left( x\right) -\frac{13}{4}\left( 1+\ln \left( 2n\right) \right) .
\end{eqnarray*}
We recall that $P_n(x)$ has non-negative coefficients. It is therefore sufficient
to show that
\begin{equation}
P_{n}\left( 2\right) \geq \frac{13}{4}\left( 1+\ln \left( 2n\right) \right) 
\label{eq:n1}
\end{equation}
for $n\geq 1$. Since $P_n(2)> P_n(1)$ we 
can now use (\ref{BO2}). 
%%%newpage
\newpage
Suppose $n\geq 81$ then
\begin{eqnarray*}
P_{n}\left( 1\right) &>&\frac{\sqrt{3}}{12n}\left( 1-\frac{1}{\sqrt{n}}\right) \mathrm{e}^{\frac{\pi }{6}\sqrt{24n-1}} \\
&>&\frac{\sqrt{3}}{4!12n} \left( 1 - \frac{1}{\sqrt{n}}\right) \left( \frac{\pi}{6} \sqrt{24n-1}\right) ^{4} \\
&=&\frac{\sqrt{3}}{4!12}\frac{24n-1}{n}\left( 1-\frac{1}{\sqrt{n}}\right) \left( \frac{\pi }{6}\right) ^{4}\left( 24n-1\right) \\
&>&\frac{23\sqrt{3}}{4!12}\frac{8}{9}\left( \frac{\pi }{6}\right) ^{4}\left( 24n-1\right) =f\left( n\right) .
\end{eqnarray*}
Now
$  13/14  \left( 1+\ln \left( 2n\right)\right)  \leq     13/14      \left( \ln \left( 162\right) +\frac{n}{81} \right)           \leq f\left( n\right) $
if  and  only if
\[
n\geq 87>\left( \frac{23\sqrt{3}}{4!12}\frac{8}{9}\left( \frac{ \pi }{6}\right) ^{4}+\frac{13}{4}\left( \ln \left( 162\right) +\frac{1}{81}\right) \right) /\left( \frac{23\sqrt{3}}{12}\frac{8}{9}\left( \frac{\pi }{6}\right) ^{4}-\frac{13}{324}\right) .
\]
Thus for $n\geq 87$ holds
$P_{n}\left( x\right) >P_{n}\left( 1\right) >
13/14 \left(\ln \left( 162\right) +\frac{n}{81} \right)         \geq 13/14 \left( 1+\ln \left( 2n\right) \right)$.
It remains to check
$P_{n}\left( 2\right) >\frac{13}{4}\left( 1+\ln \left( 2n\right) \right) $
for $1\leq n\leq 86$. In this case we have
$\left( 1+\ln \left( 2n\right) \right) <7$.
We have
\[
\begin{array}{|c||c|c|c|c|c|}
\hline
n & 1 & 2 & 3 & 4 & 5 \\ \hline
P_{n}\left( 2\right) & 2 & 5 & 10 & 20 & 36 \\ \hline
\end{array}
\]
and since $P_{n}\left( 2\right) $ increases monotonously in $n$ by
Theorem \ref{monotonie} we also have
$\frac{13}{4}\left( 1+\ln \left( 2n\right) \right)  <36=P_{5}\left( 2\right) \leq P_{n}\left( 2\right) $
for $5\leq n\leq 86$.
\end{proof}
\begin{remark}
A sharper estimation of the left
hand side of (\ref{BO2}) could
show that it
is already larger than
$\frac{13}{4}\left( 1+\ln \left( 2n\right) \right) $
for $n\geq 8$.
\end{remark}

We deduce from the proof of the
proposition the following property.

\begin{corollary}
\label{hilfsresultat}Let $n \geq 2$ and $x>2$. Then 
\begin{equation}
P_n(x) - \left(1+\ln \left( 2n\right)\right) >0.
\label{eq:hilfsresultat}
\end{equation}
\end{corollary}
Let $n \geq 4$. Then ( \ref{eq:hilfsresultat}) is satisfied for all $x \geq 1$.
%%
%%
%%
%%
%%%%%%%%%%%%%%%%%%%%%%%
\section{General case}
\begin{proof}[Proof of Theorem \ref{Hauptsatz}]
We show (\ref{conjecture}) by induction on $n=a+b$.
For $n=a+b=3$ we have
$P_{1}\left( x\right) P_{2}\left( x\right) -P_{3}\left( x\right) =\frac{x}{3}\left( x^{2}-4\right) >0$
for all $x>2$.

Suppose $n\geq 4$ and
$P_{A}\left( x\right) P_{B}\left( x\right) >P_{A+B}\left( x\right) $
for all $3\leq A+B\leq n-1=a+b-1$ and $x>2$.
Without loss of generality
we assume $a\geq b\geq 2$. (The case $b=1$ was discussed in
Proposition ~\ref{monotonie}.)

We have
$P_{a}\left( 2\right) P_{b}\left( 2\right) \geq P_{a+b}\left( 2\right) $
for $a+b\geq 3$ by Theorem~\ref{Fu} of \cite{CFT18}.
If we can now show that
\begin{equation}
\frac{\mathrm{d}}{\mathrm{d}x}\left( P_{a}\left( x\right) P_{b}\left( x\right) \right) >P_{a+b}^{\prime }\left( x\right) 
\label{eq:ableitung}
\end{equation}
the proof is completed, as this implies
$P_{a}\left( x\right) P_{b}\left( x\right) >P_{a+b}\left( x\right) $
for all $x>2$.

Note that
$P_{A}\left( x\right) P_{0}\left( x\right) =P_{A}\left( x\right) $.
Thus
\begin{eqnarray*}
&&P_{a}^{\prime }\left( x\right) P_{b}\left( x\right) +P_{a}\left( x\right) P_{b}^{\prime }\left( x\right) -P_{a+b}^{\prime }\left( x\right) \\
&=&\sum _{k=1}^{a}\frac{\sigma \left( k\right) }{k}P_{a-k}\left( x\right) P_{b}\left( x\right) +P_{a}\left( x\right) \sum _{k=1}^{b}\frac{\sigma \left( k\right) }{k}P_{b-k}\left( x\right) -\sum _{k=1}^{a+b}\frac{\sigma \left( k\right) }{k}P_{a+b-k}\left( x\right) \\
&>&\sum _{k=1}^{a}\frac{\sigma \left( k\right) }{k}P_{a+b-k}\left( x\right)+\sum _{k=1}^{b}\frac{\sigma \left( k\right) }{k}P_{a+b-k}\left( x\right) -\sum _{k=1}^{a+b}\frac{\sigma \left( k\right) }{k}P_{a+b-k}\left( x\right) \\
&=&\sum _{k=1}^{b}\frac{\sigma \left( k\right) }{k}P_{a+b-k}\left( x\right) -\frac{\sigma \left( k+a\right) }{k+a}P_{b-k}\left( x\right) \\
&\geq &\sum _{k=1}^{b}P_{a+b-k}\left( x\right) -\left( 1+\ln \left( 2a\right) \right) P_{b-k}\left( x\right) .
\end{eqnarray*}
We now consider
\begin{equation}
P_{a+b-k}\left( x\right) -\left( 1+\ln \left( 2a\right) \right) P_{b-k}\left( x\right)
\label{eq:summand}
\end{equation}
for each $k$ separately.
From Theorem~\ref{HNT Theorem} we know that
$P_{a+b-k}\left( x\right) $ increases faster than
$P_{b-k}\left( x\right) $ for $x\geq 1$.
Hence to show that (\ref{eq:summand}) is positive it is
enough to show this for $1\leq x\leq 2$.

Using (\ref{BO2})
for $k<b$
\begin{eqnarray*}
&&P_{a+b-k}\left( 1\right) -\left( 1+\ln \left( 2a\right) \right) P_{b-k}\left( 1\right) \\
&>&\frac{\sqrt{3}}{12\left( a+b-k\right) }\left( 1-\frac{1}{\sqrt{a+b-k}}\right) \mathrm{e}^{\frac{\pi }{6}\sqrt{24\left( a+b-k\right) -1}} \\
&&{}-\left( 1+\ln \left( 2a\right) \right) \frac{\sqrt{3}}{12\left( b-k\right) }\left( 1+\frac{1}{\sqrt{b-k}}\right) \mathrm{e}^{\frac{\pi }{6}\sqrt{24\left( b-k\right) -1}}
\end{eqnarray*}
and the last is positive if and only if
\begin{equation}
\mathrm{e}^{\frac{\pi }{6}\left( \sqrt{24\left( a+b-k\right) -1}-\sqrt{24\left( b-k\right) -1}\right) }
>\left( 1+\ln \left( 2a\right) \right) \frac{a+b-k}{b-k}\frac{1+\frac{1}{\sqrt{b-k}}}{1-\frac{1}{\sqrt{a+b-k}}}.
\label{eq:zuzeigen}
\end{equation}
Now
\begin{eqnarray*}
&&\sqrt{24\left( a+b-k\right) -1}-\sqrt{24\left( b-k\right) -1} \\
&=&\frac{24a}{\sqrt{24\left( a+b-k\right) -1}+\sqrt{24\left( b-k\right) -1}} \\
&\geq &\frac{24a}{\sqrt{24\left( 2a-1\right) -1}+\sqrt{24\left( a-1\right) -1}} \\
&>&\frac{24a}{7\sqrt{a}+5\sqrt{a}}=2\sqrt{a}.
\end{eqnarray*}
On the other hand
\[
\left( 1+\ln \left( 2a\right) \right) \frac{a+b-k}{b-k}\frac{1+\frac{1}{\sqrt{b-k}}}{1-\frac{1}{\sqrt{a+b-k}}}
<\left( 1+\ln \left( 2a\right) \right) \left( 1+a\right) \frac{2}{1-\frac{1}{\sqrt{a}}}.
\]
If $a$ fulfills
\begin{equation}
\mathrm{e}^{\pi \sqrt{a}/3}>\left( 1+\ln \left( 2a\right) \right) \left( 1+a\right) \frac{2}{1-\frac{1}{\sqrt{a}}}
\label{eq:wachstum}
\end{equation}
then $a$ also fulfills (\ref{eq:zuzeigen}).
It is not difficult to
show that (\ref{eq:wachstum}) is fulfilled
for $a\geq 34$.
This implies (\ref{eq:summand}) for $x=1$.
Theorem~\ref{HNT Theorem} implies (\ref{eq:summand})
for $x\geq 1$.

We used PARI/GP to check that equation (\ref{eq:summand})
with $x=2$ is positive
for $1\leq k<b\leq a\leq 33$.
Again Theorem~\ref{HNT Theorem}
implies (\ref{eq:summand}) for $x\geq 2$.

What remains to be considered
is (\ref{eq:summand}) for $k=b$.
This follows from
Corollary~\ref{hilfsresultat}
setting $m=a$.

Thus, for all $a\geq b\geq k\geq 1$
the value of (\ref{eq:summand}) is positive. Hence  also
\[
P_{a}^{\prime }\left( x\right) P_{b}\left( x\right) +P_{a}\left( x\right) P_{b}^{\prime }\left( x\right) >P_{a+b}^{\prime} \left( x\right) 
\]
for $x>2$. This completes the induction step,
as explained  in the  beginning of the proof.
\end{proof}

In the following we give provide a concise proof
of (\ref{eq:wachstum}) for
$a\geq 85$. 
All that remains is to check that (\ref{eq:wachstum})
is fulfilled for $a\leq 84$. The sharp
estimate $a\geq 34$ would require more
work.

\begin{proof}[Analytic proof of (\ref{eq:wachstum}) for $a\geq 85$]
For $a>0$
\[
\mathrm{e}^{\pi \sqrt{a}/3}>\frac{1}{6!}\left( \frac{\pi }{3}\sqrt{a}\right) ^{6}=\frac{1}{6!}\left( \frac{\pi }{3}\right) ^{6}a^{3}.
\]
Suppose $a\geq 81$ then
\[
\left( 1+\ln \left( 2a\right) \right) \left( 1+a\right) \frac{2}{1-\frac{1}{\sqrt{a}}}\leq \left( \frac{a}{81}+\ln \left( 162\right) \right) \left( 1+a\right) \frac{18}{8}.
\]
Thus (\ref{eq:wachstum}) is fulfilled if
\[
\frac{1}{6!}\left( \frac{\pi }{3}\right) ^{6}a\geq \left( \frac{1}{81}+\frac{\ln \left( 162\right) }{81}\right) \left( \frac{1}{81}+1\right) \frac{18}{8}\geq\left( \frac{1}{81}+\frac{\ln \left( 162\right) }{a}\right) \left( \frac{1}{a}+1\right) \frac{9}{4}
\]
and this is fulfilled for
$a\geq 85>\left( \left( \frac{1}{81}+\frac{\ln \left( 162\right) }{81}\right) \left( \frac{1}{81}+1\right) \frac{9}{4}\right) /\left( \frac{1}{6!}\left( \frac{\pi }{3}\right) ^{6}\right) $.
\end{proof}
%%%%%%%%%%%%%%%%%%%%%%%
%%
{\bf Acknowledgments.} 
%The authors thank the two reviewers for their very helpful suggestions.
We thank the RWTH Aachen University
and the Graduate School: Experimental and constructive algebra
%% and the German University of Technology in Oman 
for their excellent working environment.
%%\address{German University of Technology in Oman, Muscat, Sultanate of Oman}

%%%%%%%%%%%%%%%%%%%%%%
%%%%%%%%%%%%%%%%%%%%%%
\end{document}